\documentclass[a4paper]{amsart}
\usepackage[all]{xy}

\newcommand{\oo}{\mathcal{O}}
\newcommand{\p}{\mathbb{P}}

\newcommand{\e}{\mathcal{E}}
\newcommand{\nn}{\mathcal{N}}
\newcommand{\ff}{\mathcal{F}}
\newcommand{\q}{\mathcal{Q}}

\newcommand{\ggg}{\mathcal{G}}

\newcommand{\ext}{\rm{Ext}}
\newcommand{\rt}{\rightarrow}

\newtheorem{thm}{Theorem}[section]
\newtheorem{prop}[thm]{Proposition}
\newtheorem{cor}[thm]{Corollary}
\newtheorem{lem}{Lemma}
\newtheorem*{prop*}{Proposition}
\newtheorem*{thm*}{Theorem}

\newtheorem*{claim*}{Claim}

\theoremstyle{definition}

\theoremstyle{remark}
\newtheorem{rem}{Remark}

\theoremstyle{remark}

\begin{document}
\title{On globally generated vector bundles on projective spaces II}
\author[Jos\'e Carlos Sierra and Luca Ugaglia]{Jos\'e Carlos Sierra* and Luca Ugaglia}

\address{Instituto de Ciencias Matem\'aticas (ICMAT), Consejo Superior de Investigaciones Cient\'{\i}ficas (CSIC), Campus de Cantoblanco, 28049 Madrid, Spain}
\email{jcsierra@icmat.es}

\address{Dipartimento di Matematica e Informatica, Universit\`a degli Studi di Palermo, Via Archirafi 34,  90123, Palermo, Italy}
\email{luca.ugaglia@unipa.it}

\thanks {* Research supported by the ``Ram\'on y Cajal" contract
RYC-2009-04999, the project MTM2009-06964 of MICINN and the ICMAT ``Severo Ochoa" project SEV-2011-0087 of MINECO}

\begin{abstract}
Extending the main result of \cite{su2}, we classify globally generated vector bundles on $\p^n$ with first Chern class equal to $3$.
\end{abstract}
\maketitle

\section{Main result}

The main result of the paper is the following.

\begin{thm} \label{thm:main}
Let $\e$ be a globally generated vector bundle of rank $k$ on
$\p^n$. If $c_1(\e)=3$ and $c_2(\e)\leq 4$ then one of the following
holds:
\begin{enumerate}
\item [(i)] $c_2(\e)=0$ and $\e=\oo_{\p^n}(3)$;
\item [(ii)] $c_2(\e)=2$ and $\e=\oo_{\p^n}(2)\oplus\oo_{\p^n}(1)$;
\item [(iii)] $c_2(\e)=3$ and $\e=\oo_{\p^n}(2)\oplus T_{\p^n}(-1)$;
\item [(iv)] $c_2(\e)=3$ and $\e=\oo_{\p^n}(1)\oplus\oo_{\p^n}(1)\oplus\oo_{\p^n}(1)$;
\item [(v)] $c_2(\e)=4$ and $\e=\oo_{\p^n}(1)\oplus\oo_{\p^n}(1)\oplus T_{\p^n}(-1)$;
\item [(vi)] $c_2(\e)=4$ and $\e=\oo_{\p^3}(1)\oplus\Omega_{\p^3}(2)$;
\item [(vii)] $c_2(\e)=4$ and $\e=\Omega_{\p^4}(2)$;
\item [(viii)] $\e$ is given by an exact sequence $0\to\oo_{\p^n}^{\oplus s}\to\ggg\oplus\oo_{\p^n}^{\oplus r}\to\e\to 0$, where $h^0(\e^*)=r$, $h^1(\e^*)=s$ and $\ggg$ is a bundle as above.
\end{enumerate}
\end{thm}

Theorem \ref{thm:main} immediately implies the following corollary.

\begin{cor} \label{cor:main}
Let $\e$ be a globally generated vector bundle of rank $k$ on
$\p^n$. If $c_1(\e)=3$ then $\e$ is either as in Theorem
\ref{thm:main}, or one of the following holds:
\begin{enumerate}
\item [(i)] $c_2(\e)=5$ and $\e=\Omega_{\p^4}^2(2)^*$;
\item [(ii)] $c_2(\e)=5$ and $\e=T_{\p^3}(-1)\oplus\Omega_{\p^3}(2)$;
\item [(iii)] $c_2(\e)=5$ and $\e=T_{\p^n}(-1)\oplus T_{\p^n}(-1)\oplus\oo_{\p^n}(1)$;
\item [(iv)] $c_2(\e)=6$ and $\e=T_{\p^n}(-1)\oplus T_{\p^n}(-1)\oplus T_{\p^n}(-1)$;
\item [(v)] $c_2(\e)=6$ and $0\rt\oo_{\p^n}(-2)\oplus\Omega_{\p^n}(1)\rt \oo_{\p^n}^{\oplus k+n+1}\rt\e\rt 0$;
\item [(vi)] $c_2(\e)=7$ and $0\rt\oo_{\p^n}(-2)\oplus\oo_{\p^n}(-1)\rt \oo_{\p^n}^{\oplus k+2}\rt\e\rt 0$;
\item [(vii)] $c_2(\e)=9$ and $0\rt\oo_{\p^n}(-3)\rt\oo_{\p^n}^{\oplus k+1}\rt\e\rt 0$;
\item [(viii)] $\e$ is given by $0\to\oo_{\p^n}^{\oplus s}\to\ggg\oplus\oo_{\p^n}^{\oplus r}\to\e\to 0$, where $h^0(\e^*)=r$, $h^1(\e^*)=s$ and $\ggg$ is
a bundle as above.
\end{enumerate}
\end{cor}

This note is a natural extension of \cite{su2}. Therefore, we also want to thank the referee of that paper for his help.

Globally generated vector bundles $\e$ on $\p^n$ with $c_1(\e)=3$ have also been studied independently, and using a different approach, in \cite{man} and \cite{a-m} (cf. Remark \ref{rem:hidden}).

\section{Proof of Theorem \ref{thm:main}}\label{section:proof}

We work over the field of complex numbers. Let $\e$ be a globally
generated vector bundle on $\p^n$ of rank $k$, and let $\e^*$ denote
its dual bundle. In view of the following result, we will assume
throughout the paper that $h^0(\e^*)=h^1(\e^*)=0$.

\begin{lem}[First reduction]\label{lem:1st reduction}
Let $\e$ be a globally generated vector bundle on $\p^n$. If
$h^0(\e^*)=r$ and $h^1(\e^*)=s$, then there exists a globally
generated vector bundle $\ggg$ such that
$h^0(\ggg^*)=h^1(\ggg^*)=0$, and an exact sequence $$0\to\oo^{\oplus
s}_{\p^n}\to\ggg\oplus\oo_{\p^n}^{\oplus r}\to\e\to 0.$$
\end{lem}

\begin{proof}
Just put together \cite[Lemmas 3 and 4]{su2}.
\end{proof}

Let $c_1:=c_1(\e)$ and $c_2:=c_2(\e)$ denote the first and second
Chern class of $\e$, respectively. We point out that $c_1^2-c_2\geq 0$ since $\e$ is globally generated. Furthermore, in order to classify
globally generated vector bundles one can assume $c_2\leq
\frac{c_1^2}{2}$ thanks to the following.

\begin{lem}[Second reduction]\label{lem:2nd reduction}
Let $\e$ be a globally generated vector bundle with Chern classes
$c_1, c_2$. If $c_2>\frac{c_1^2}{2}$ then there exists a globally
generated vector bundle $\mathcal K^*$, whose dual $\mathcal K$ is
given by the exact sequence
$$0\to\mathcal K\to\oo_{\p^n}^{\oplus h^0(\e)}\to\e\to 0.$$ In particular, $c_1(\mathcal
K^*)=c_1$ and $c_2(\mathcal K^*)=c_1^2-c_2<\frac{c_1^2}{2}$.
\end{lem}

\begin{proof}
Consider the kernel $\mathcal K$ of the epimorphism
$\oo_{\p^n}^{\oplus h^0(\e)}\to \e\to 0$.
\end{proof}

Globally generated vector bundles with $c_1\leq 2$ were classified
in \cite{su2}. From now on we concentrate on the case $c_1=3$. We
start by considering the cases in which $\e$ admits a global section
whose zero locus is a hypersurface in $\p^n$.

\begin{prop}\label{prop:-3}
If $h^0(\e(-3))\neq 0$ then $\e=\oo_{\p^n}(3)$. Moreover, if
$h^0(\e(-3))=0$ then $h^0(\e_K(-3))=0$ for every linear subspace
$K\subset\p^n$ of dimension greater than 1.
\end{prop}

\begin{proof}
The first statement was shown in \cite[Lemma 5]{su2}. On the other
hand, if $h^0(\e_K(-3))\neq 0$ then
$\e_K=\oo_K(3)\oplus\oo_K^{\oplus k-1}$ by the first assertion and
Lemma \ref{lem:1st reduction}. Therefore
$\e=\oo_{\p^n}(3)\oplus\oo_{\p^n}^{\oplus k-1}$ (see for instance \cite[Ch.~I, Theorem 2.3.2]{oko}), whence $h^0(\e(-3))\neq 0$.
\end{proof}

\begin{prop}\label{prop:-2}
Assume that $h^0(\e(-3))=0$.
\begin{enumerate}
\item[(i)] If $h^0(\e(-2))\neq 0$ then either
$\e=\oo_{\p^n}(2)\oplus\oo_{\p^n}(1)$, or $\e=\oo_{\p^n}(2)\oplus
T_{\p^n}(-1)$.

\item[(ii)] If $h^0(\e(-2))=0$ then $h^0(\e_K(-2))=0$ for every linear subspace
$K\subset\p^n$ of dimension greater than 1.
\end{enumerate}
\end{prop}

\begin{proof}
To prove (i), we essentially argue as in \cite[Proposition
3.2]{su2}. If $n=1$ the result is trivial, so we assume that $n\geq 2$.
Let $s\in H^0(\e(-2))$ be a non-zero section. Consider the
corresponding exact sequence of sheaves
$$0\to\oo_{\p^n}\to\e(-2)\to\ff\to 0,$$ and let $Z\subset\p^n$ be the
zero locus of $s$. We claim that $Z$ is a finite scheme of length at
most 1. To get a contradiction, let $P,Q$ be two points (maybe
infinitely close) where $s$ vanishes and let $L\subset\p^n$ be the
line joining $P$ and $Q$. Restricting to $L$ and twisting, we get
$$0\to\oo_L(2)\to\e_L\to\ff_L(2)\to 0.$$ Since $\e$ is globally
generated and $\ff_L(2)$ is a quotient of $\e_L$, we deduce that
$\ff_L(2)$ is also globally generated. Furthermore
$$3=c_1(\e_L)=c_1(\oo_L(2))+c_1(\ff_L(2))=2+c_1(\ff_L(2))$$ and $P,Q\in Z$, so $s$
vanishes on $L$. Let $\p^2\subset\p^n$ be a general plane containing
$L$. Then $s$ does not vanish identically on $\p^2$ as otherwise
$s\in H^0(\e(-2))$ would be the zero section. Let $V\subset\p^n$ be
the hypersurface of degree $2$ where $s$ vanishes (consi\-dered as a
section of $\e$). Then $s$ vanishes on $L$ and $V\cap\p^2$, whence
$h^0(\e_{\p^2}(-3))\neq 0$ contradicting Proposition \ref{prop:-3}.
This proves the claim. Consider the restriction sequence
$$0\to\oo_H\to\e_H(-2)\to\ff_H\to 0$$ to a hyperplane
$H\subset\p^n$ not meeting $Z$. Then $\ff_H$ is a vector bundle such
that $\ff_H(2)$ is globally generated and $c_1(\ff_H(2))=1$.
Therefore $\ff_H(2)$ is either $\oo_H(1)\oplus\oo_H^{\oplus k-2}$ or
$T_H(-1)\oplus\oo_H^{\oplus k-n}$ by \cite[Proposition 3.1]{su2}. As
$$\e_H(-2)\in{\ext}^1(\ff_H,\oo_H)=H^{n-2}(\ff_H(-n))=0,$$ we deduce that
$\e_H$ is either $\oo_H(2)\oplus\oo_H(1)\oplus\oo_H^{\oplus k-2}$ or $\oo_H(2)\oplus T_H(-1)\oplus\oo_H^{\oplus k-n}$.
We claim that $h^{n-1}(\ff(-n-1))=0$. Assume that the claim is proved. Then
$$\e(-2)\in{\ext}^1(\ff,\oo_{\p^n})=H^{n-1}(\ff(-n-1))=0,$$ whence
$\e(-2)=\oo_{\p^n}\oplus\ff$ and $\ff$ is a vector bundle such that
$\ff(2)$ is globally generated and $c_1(\ff(2))=1$, so $\e$ is
either $\oo_{\p^n}(2)\oplus\oo_{\p^n}(1)$ or $\oo_{\p^n}(2)\oplus
T_{\p^n}(-1)$ by \cite[Proposition 3.1]{su2} and Lemma \ref{lem:1st
reduction}. Let us prove the claim. Consider the restriction sequence
$$0\to\e^*\to\e^*(1)\to\e_H^*(1)\to 0.$$ Since $h^1(\e^*)=0$ by
assumption and $h^1(\e_H^*(1))=0$, we
get $h^1(\e^*(1))=0$. Now consider the restriction sequence $$0\to\ff(-n-1)\to\ff(-n)\to\ff_H(-n)\to 0.$$
As $h^{n-2}(\ff_H(-n))=0$ and $h^{n-1}(\ff(-n))=h^{n-1}(\e(-n-2))=h^1(\e^*(1))=0$, we deduce that $h^{n-1}(\ff(-n-1))=0$.

We now prove (ii). It suffices to show it for every hyperplane
$H\subset\p^n$. If $h^0(\e_H(-2))\neq 0$ for some hyperplane
$H\subset\p^n$, we deduce from (i) and Lemma \ref{lem:1st reduction}
that either $\e_H=\oo_H(2)\oplus\oo_H(1)\oplus\oo_H^{\oplus k-2}$,
or $\e_H$ fits in an exact sequence
$$0\to\oo_H^{\oplus s}\to\oo_H(2)\oplus T_H(-1)\oplus\oo_H^{\oplus k+s-n}\to \e_H\to 0$$
Assume first that $n\geq 4$. Then $h^i(\e_H(-j))=0$ for $i=0,1$ and every $j\geq 3$. Consider the restriction sequence
$0\to\e(-j-1)\to\e(-j)\to\e_H(-j)\to 0$. We deduce from Serre's
vanishing theorem that $h^0(\e(-3))=h^1(\e(-3))=0$, whence
$h^0(\e(-2))=h^0(\e_H(-2))\neq 0$. Now, assume that $n=3$.
The Hirzebruch-Riemann-Roch theorem yields $\chi(\e)=\frac{1}{6}(c_1^3-3c_1c_2+3c_3)+c_1^2-2c_2+\frac{22}{12}c_1+k$.
Since $c_1=c_2=3$ we deduce that $\chi(\e)=8+k+\frac{1}{2}(c_3+1)$. To get a contradiction, assume that $h^0(\e(-2))=0$.
Then the restriction sequence gives $h^0(\e(-1))\leq 3$ and $h^0(\e)\leq 9+k$. We deduce that
$h^3(\e)=h^0(\e^*(-4))=0$ and $h^2(\e)=h^1(\e^*(-4))=0$ from Serre duality and the exact sequence
$0\to\e^*(-j-1)\to\e^*(-j)\to\e_H^*(-j)\to 0$. Hence $-h^1(\e)\geq (c_3-1)/2$, that is, $c_3=1$ since $c_3$
is odd (see, for instance, \cite[p. 113]{oko}). Therefore $\e=\oo_{\p^3}(1)^{\oplus 3}$, and we get a contradiction.
\end{proof}

\begin{rem}
We would like to thank Edoardo Ballico for pointing out the following gap in the proof of \cite[Proposition 3.2]{su2}. The natural isomorphism between $H^{n-1}(\ff(-n-1))$ and the dual of $H^1(\ff^*)$ holds if the quotient $\ff$ is a locally free sheaf, so we just have $\e(-1)\in{\ext}^1(\ff,\oo_{\p^n})=H^{n-1}(\ff(-n-1))$. In order to show that $h^{n-1}(\ff(-n-1))=0$, and hence $\e(-1)=\ff\oplus\oo_{\p^n}$, just note that $h^{n-1}(\ff(-n))=h^{n-1}(\e(-n-1))=h^1(\e^*)=0$ and that $h^{n-2}(\ff_H(-n))=0$ (cf. Lemma \ref{lem:F bundle} below).
\end{rem}

The cases $h^0(\e(-3))\neq 0$ and $h^0(\e(-2))\neq 0$ were described
in Propositions \ref{prop:-3} and \ref{prop:-2}, respectively. Now we study in detail the case $h^0(\e(-1))\neq 0$. 

\begin{lem}\label{lem:F bundle}
Let $s\in H^0(\e(-1))$ be a non-zero section, and let
$0\to\oo_{\p^n}\to\e(-1)\to\ff\to 0$ be the corresponding exact
sequence of sheaves. If $h^{n-2}(\ff_H(-n))=0$ for some hyperplane
$H\subset\p^n$ then $\e=\oo_{\p^n}(1)\oplus\ff(1)$. In particular,
$\ff(1)$ is a globally generated vector bundle with $c_1(\ff(1))=2$.
\end{lem}

\begin{proof}
We deduce that $h^{n-1}(\ff(-n))=h^{n-1}(\e(-n-1))=h^1(\e^*)=0$ from the exact sequence $0\to\oo_{\p^n}(-n)\to\e(-n-1)\to\ff(-n)\to 0$, Serre duality, and the assumption that $h^1(\e^*)=0$ throughout the paper. Therefore, we get $h^{n-1}(\ff(-n-1))=0$ from the restriction sequence $0\to\ff(-n-1)\to\ff(-n)\to\ff_H(-n)\to 0$. As $\e(-1)\in{\ext}^1(\ff,\oo_{\p^n})=H^{n-1}(\ff(-n-1))=0$, we deduce that $\e(-1)=\oo_{\p^n}\oplus\ff$.
\end{proof}

From now on, we also assume that $c_2\leq 4$ (cf. Lemma \ref{lem:2nd
reduction}).

\begin{prop}\label{prop:e(-1)}
If $h^0(\e(-2))=0$ and $h^0(\e(-1))\neq 0$ then one of the following
holds:
\begin{itemize}
\item $\e=\oo_{\p^n}(1)\oplus\oo_{\p^n}(1)\oplus\oo_{\p^n}(1)$;
\item $\e=\oo_{\p^n}(1)\oplus\oo_{\p^n}(1)\oplus T_{\p^n}(-1)$;
\item $\e=\oo_{\p^3}(1)\oplus\Omega_{\p^3}(2)$.
\end{itemize}
\end{prop}

\begin{proof}
For $n=1$, the result is obvious, so we assume that $n\geq 2$. Let $s\in
H^0(\e(-1))$ be a non-zero section, and consider the exact sequence
of sheaves
$$0\to\oo_{\p^n}\to\e(-1)\to\ff\to 0.$$ Let
$Z\subset\p^n$ be the zero locus of $s$. We claim that $Z$ is a
finite scheme of length at most $2$. To get a contradiction, let
$T\subset Z$ be a subscheme of length $3$ and let $\Pi\subset\p^n$
be a plane containing $T$. Consider the restriction $\e_{\Pi}$ and
the quotient $$0\to\oo_{\Pi}^{\oplus k-2}\to\e_{\Pi}\to\q\to 0$$
(cf. \cite[Ch.~I, Lemma 4.3.1]{oko}). Then $\q$ is a globally
generated vector bundle of rank $2$, $c_1(\q)=c_1(\e_{\Pi})=3$ and
$c_2(\q)=c_2(\e_{\Pi})\leq 4$. The restriction to ${\Pi}$ of the
non-zero section $s\in H^0(\e(-1))$ yields a non-zero section in
$H^0(\e_{\Pi}(-1))$ by Proposition \ref{prop:-2}(ii). Therefore,
since $H^0(\e_{\Pi}(-1))\cong H^0(\q(-1))$, we get a non-zero
section $\sigma \in H^0(\q(-1))$ vanishing on $T\subset\Pi$. Since
the zero locus of $\sigma$ is finite as otherwise $\sigma\in
H^0(\q(-2))\cong H^0(\e_{\Pi}(-2))=0$, we get
$c_2(\q(-1))\geq 3$ contradicting the fact that
\[
c_2(\q(-1))=(-1)^2-c_1(\q)+c_2(\q)=c_2(\q)-2\leq 2,
\]
and hence proving the claim. Now consider the restriction
$$0\to\oo_H\to\e_H(-1)\to\ff_H\to 0$$ to a hyperplane $H\subset\p^n$ such that
$Z\cap H=\emptyset$. Then $\ff_H(1)$ is a globally generated vector
bundle, $c_1(\ff_H(1))=2$ and $c_2(\ff_H(1))\leq 2$. Therefore
$\ff_H(1)$ can be as in \cite[Theorem 1.1]{su2}, cases (i)-(iv). In
case (i) we have $\ff_H(1)=\oo_H(2)\oplus\oo_H^{\oplus k-2}$, so
$h^{n-2}(\ff_H(-n))=0$ and hence $\e=\oo_{\p^n}(1)\oplus\oo_{\p^n}(2)$ by
Lemma \ref{lem:F bundle}, giving a contradiction. In case (ii) we
have $\ff_H(1)=\oo_H(1)^{\oplus 2}\oplus\oo_H^{\oplus k-3}$, so
$h^{n-2}(\ff_H(-n))=0$ and therefore
$\e=\oo_{\p^n}(1)\oplus\oo_{\p^n}(1)\oplus\oo_{\p^n}(1)$ by Lemma
\ref{lem:F bundle}. In case (iv) we also have $h^{n-2}(\ff_H(-n))=0$.
Therefore, by Lemma
\ref{lem:F bundle}, $\ff(1)$ is a globally generated vector bundle such that $\ff_H(1)$ is either
$\Omega_{\p^3}(2)\oplus\oo_{\p^3}^{\oplus k-4}$ or
$\nn(1)\oplus\oo_{\p^3}^{\oplus k-3}$, and we get a contradiction by
\cite[Theorem 1.1]{su2}. If $\ff_H(1)$ is as in case (iii), we
remark that $\ff_H(1)$ is either
$T_H(-1)\oplus\oo_H(1)\oplus\oo_H^{k-n-1}$ or
$\ggg\oplus\oo_H^{k-n}$, where $\ggg$ is a vector bundle of rank
$n-1$ obtained as a quotient $$0\to\oo_H\to\oo_H(1)\oplus
T_H(-1)\to\ggg\to 0$$ (cf. \cite[Remark 3]{su2}). If
$\ff_H(1)=T_H(-1)\oplus\oo_H(1)\oplus\oo_H^{k-n-1}$ then
$h^{n-2}(\ff_H(-n))=0$, and hence $\e=\oo_{\p^n}(1)\oplus\ff(1)$ by Lemma
\ref{lem:F bundle}. Therefore, $\ff(1)$ is either
$T(-1)\oplus\oo_{\p^n}(1)$ or $\Omega_{\p^3}(2)$ by \cite[Theorem
1.1]{su2}. Let us see now that $\ff_H(1)=\ggg\oplus\oo_H^{k-n}$
yields a contradiction. Assume first that $n=3$. Then $h^1(\ff_H(-3))=h^1(\ggg(-4))=h^1(\ggg^*(-2))=0$, as
$\ggg^*(-2)\cong\ggg(-4)$ since $\ggg$ is of rank $2$ and
$c_1(\ggg(-4))=-6$. Therefore, by Lemma
\ref{lem:F bundle}, $\e=\oo_{\p^3}(1)\oplus\ff(1)$, and hence
$\ff(1)=\nn_{\p^3}(1)\oplus\oo_{\p^3}^{\oplus k-3}$ by \cite[Theorem
1.1]{su2}. This contradicts the assumption that $h^1(\e^*)=0$. Assume now that
$n\geq 4$. To get a contradiction, we point out that
$h^1(\ff^*_H(-1))=h^1(\ggg^*)=1$. Then it follows from the exact
sequence
$$0\to\ff^*_H(-1)\to\e_H^*\to\oo_H(-1)\to 0$$ that $h^1(\e_H^*)=1$.
Hence the exact sequence $$0\to\e^*(-1)\to\e^*\to\e_H^*\to 0$$
yields $h^2(\e^*(-1))\neq 0$, as we assume that $h^1(\e^*)=0$. Let us see
that $h^2(\e^*(-2))=0$. Consider the exact sequence
$$0\to\ff^*_H(-1-j)\to\e_H^*(-j)\to\oo_H(-1-j)\to 0.$$ Then
$h^i(\e_H^*(-j))=h^i(\ff^*_H(-1-j))=h^i(\ggg^*(-j))=0$ for
$i\in\{1,2\}$ and every integer $j\geq 2$ (here we use $n\geq 4$).
So we deduce from the exact sequence
$$0\to\e^*(-1-j)\to\e^*(-j)\to\e_H^*(-j)\to 0$$ and Serre's
vanishing theorem that $h^2(\e^*(-2))=0$. Therefore,
$h^2(\e^*(-2))=0$ and $h^2(\e^*(-1))\neq 0$ yields that
$h^2(\e_H^*(-1))\neq 0$, which is a contradiction as
$h^2(\e_H^*(-1))=h^2(\ff_H^*(-2))=h^2(\ggg^*(-1))=0$.
\end{proof}

Finally, we consider the case $h^0(\e(-1))=0$.

\begin{cor}\label{cor:n=4}
Assume that $n\geq 3$. If $h^0(\e(-1))=0$ but $h^0(\e_H(-1))\neq 0$ for
some hyperplane $H\subset\p^n$, then $n=4$ and $\e_{\p^3}$ is either
$\oo_{\p^3}(1)\oplus\Omega_{\p^3}(2)$, or a quotient
$0\to\oo_{\p^3}\to\oo_{\p^3}(1)\oplus\Omega_{\p^3}(2)\to\e_{\p^3}\to
0$ of rank $3$.
\end{cor}

\begin{proof}
Suppose first that $n\geq 4$. If $h^0(\e_H(-1))\neq 0$ then it follows from Lemma \ref{lem:1st
reduction} and Proposition \ref{prop:e(-1)} that $\e_H$ fits in an
exact sequence $0\to\oo_H^{\oplus s}\to\ggg\oplus\oo_H^{\oplus
r}\to\e_H\to 0$, where $r=h^0(\e^*_H)$, $s=h^1(\e^*_H)$ and either

\begin{enumerate}
\item[(i)] $\ggg=\oo_H(1)^{\oplus 3}$, or
\item[(ii)] $\ggg=\oo_H(1)^{\oplus 2}\oplus T_H(-1)$, or
\item[(iii)] $\ggg=\oo_{\p^3}(1)\oplus\Omega_{\p^3}(2)$.
\end{enumerate}

In cases (i) and (ii) we get $h^i(\e_H(-j))=h^i(\ggg(-j))=0$ for
$i\in\{0,1\}$ and every integer $j\geq 2$ (here we use $n\geq 4$).
Hence we deduce from the exact sequence
$$0\to\e(-j-1)\to\e(-j)\to\e_H(-j)\to 0$$ and Serre's vanishing
theorem that $h^0(\e(-2))=h^1(\e(-2))=0$. Therefore
$$h^0(\e(-1))=h^0(\e_H(-1))\neq 0,$$ yielding a contradiction.
Hence case (iii) holds and $n=4$. Furthermore, we claim that
$h^0(\e_H^*)=0$. From the dual sequence
$0\to\e^*_H\to\ggg^*\oplus\oo_H^{\oplus r}\to\oo_H^{\oplus s}\to 0$
we deduce that $h^i(\e_H^*(-j))=h^i(\ggg^*(-j))=0$ for $i\in\{0,1\}$
and every integer $j\geq 1$. From the exact sequence
$$0\to\e^*(-1-j)\to\e^*(-j)\to\e_H^*(-j)\to 0$$ and Serre's vanishing
theorem we get $h^0(\e^*(-1))=h^1(\e^*(-1))=0$, and hence
$h^0(\e_H^*)=h^0(\e^*)=0$. Therefore $\e_H$ is either
$\oo_{\p^3}(1)\oplus\Omega_{\p^3}(2)$, or a quotient
$0\to\oo_{\p^3}^{\oplus
s}\to\oo_{\p^3}(1)\oplus\Omega_{\p^3}(2)\to\e_H\to 0$ where, in the
latter, $s=1$ as $c_3(\oo_{\p^3}(1)\oplus\Omega_{\p^3}(2))\neq 0$.

Assume now that $n=3$. We argue as in Proposition \ref{prop:-2}. To get a contradiction, assume that $h^0(\e(-1))=0$ and $h^0(\e_H(-1))\neq 0$. Then we deduce from Proposition \ref{prop:e(-1)} that $\e_H$ is given by an exact sequence $$0\to\oo_H^{\oplus s}\to\oo_H(1)^{\oplus 2}\oplus T_H(-1)\oplus\oo_H^{\oplus k+s-4}\to \e_H\to 0$$ As $h^0(\e(-1))=0$, we deduce from the restriction sequence that $h^0(\e)\leq k+5$. We deduce that $h^3(\e)=h^0(\e^*(-4))=0$ and $h^2(\e)=h^1(\e^*(-4))=0$ from Serre duality and the exact sequence $0\to\e^*(-1-j)\to\e^*(-j)\to\e_H^*(-j)\to 0$. By the Hirzebruch-Riemann-Roch theorem, we get $\chi(\e)=\frac{1}{6}(c_1^3-3c_1c_2+3c_3)+c_1^2-2c_2+\frac{22}{12}c_1+k$, and hence $h^0(\e)-h^1(\e)=k+5+c_3/2\leq k+5-h^1(\e)$; that is, $c_3=0$ giving a contradiction (see for instance \cite[Theorem 1.1]{c-e}).
\end{proof}

Let us see that only the first case in Corollary \ref{cor:n=4}
actually occurs.

\begin{prop}\label{prop:omega}
Assume that $h^0(\e(-1))=0$ but $h^0(\e_H(-1))\neq 0$ for some hyperplane
$H\subset\p^4$. Then $\e\cong\Omega_{\p^4}(2)$.
\end{prop}

\begin{proof}
It follows from Corollary \ref{cor:n=4} that $\e_{\p^3}$ is either
$\oo_{\p^3}(1)\oplus\Omega_{\p^3}(2)$, or a quotient
$0\to\oo_{\p^3}\to\oo_{\p^3}(1)\oplus\Omega_{\p^3}(2)\to\e_{\p^3}\to
0$ of rank $3$.

If $\e_H=\oo_{\p^3}(1)\oplus\Omega_{\p^3}(2)$ then we see
from Serre's vanishing theorem and the restriction sequence
$$0\to\e(-j-1)\to\e(-j)\to\e_H(-j)\to 0$$ that
$h^1(\e(-2))=h^1(\e_H(-2))=1$. Therefore we have a non-trivial
extension $$0\to\oo_{\p^4}\to\ggg\to\e^*(2)\to 0$$ We claim that
$\ggg=\oo_{\p^4}(1)^{\oplus 5}$. In view of \cite[Ch.~I, Theorem
2.3.2]{oko}, it is enough to show that $\ggg_H=\oo_H(1)^{\oplus 5}$.
Let us see that $\ggg_H$ has no intermediate cohomology. From the
exact sequence
$$0\to\oo_H\to\ggg_H\to\oo_H(1)\oplus T_H\to 0,$$ we deduce that
$h^1(\ggg_H(j))=0$ for every integer $j$ and that $h^2(\ggg_H(j))=0$
for every integer $j\neq -4$. For $j=-4$, we have
$h^2(\ggg_H(-4))=h^1(\ggg_H^*)$. It follows from the exact sequence
$$0\to\oo_H(-1-j)\oplus\Omega_H(-j)\to\ggg_H^*(-j)\to\oo_H(-j)\to
0$$ that $h^0(\ggg_H^*(-j))=h^1(\ggg_H^*(-j))=h^2(\ggg_H^*(-j))=0$
for every $j\geq 1$. Therefore Serre's vanishing theorem applied to
the restriction sequence
$$0\to\ggg^*(-j-1)\to\ggg^*(-j)\to\ggg_H^*(-j)\to 0$$ yields
$h^1(\ggg^*(-1))=h^2(\ggg^*(-1))=0$, and hence
$h^1(\ggg_H^*)=h^1(\ggg^*)=0$. Then Horrocks' theorem (see, for instance, \cite[Ch.~I, Theorem 2.3.1]{oko}) implies that
$\ggg_H$ splits. Finally $c_1(\ggg_H)=5$ and $h^0(\ggg_H(-2))=0$, so
we get $\ggg_H=\oo_H(1)^{\oplus 5}$. Then $\e=\Omega_{\p^4}(2)$.

Assume now that $\e_H$ is given by a quotient
$$0\to\oo_{\p^3}\to\oo_{\p^3}(1)\oplus\Omega_{\p^3}(2)\to\e_{\p^3}\to
0$$ Then $c_t(\e)=c_t(\e_H)=1+3t+4t^2+2t^3$. Therefore, we get a
contradiction by the Schwarzenberger condition $(S^3_4)$
\cite[p.113]{oko} for $s=4$.
\end{proof}

We can now prove Theorem \ref{thm:main}.

\begin{proof}[Proof of Theorem \ref{thm:main}] We can assume that $h^0(\e^*)=h^1(\e^*)=0$ by Lemma
\ref{lem:1st reduction}; otherwise we get case (viii). If
$h^0(\e(-3))\neq 0$, then we get case (i) by Proposition
\ref{prop:-3}. If $h^0(\e(-3))=0$ but $h^0(\e(-2))\neq 0$, then we
get cases (ii) and (iii) by Proposition \ref{prop:-2}. If
$h^0(\e(-2))=0$ but $h^0(\e(-1))\neq 0$, then we get cases (iv), (v)
and (vi) by Proposition \ref{prop:e(-1)}. If $h^0(\e(-1))=0$ but
$h^0(\e_H(-1))\neq 0$ for some hyperplane $H\subset\p^n$, then we get
case (vii), by Corollary \ref{cor:n=4} and Proposition \ref{prop:omega}. Furthermore, we claim
that there is no vector bundle $\e$ on $\p^5$ such that
$\e_H=\Omega_{\p^4}(2)\oplus\oo_{\p^4}^{\oplus k-4}$. As
$h^i(\e^*_H(-j))=0$ for $i\in\{0,1\}$ and every integer $j\geq 1$,
we deduce from Serre's vanishing theorem and the restriction
sequence
$$0\to\e^*(-1-j)\to\e^*(-j)\to\e^*_H(-j)\to
0$$ that $h^i(\e^*(-1))=0$ for $i\in\{0,1\}$. Therefore
$h^0(\e^*)=h^0(\e_H^*)=k-4$ and hence there exists a rank-$4$ vector
bundle $\ggg$ such that $\e=\ggg\oplus\oo_{\p^5}^{\oplus k-4}$. Then
$c_t(\ggg)=c_t(\e_H)=1+3t+4t^2+2t^3+t^4$ and we get a contradiction
by the Schwarzenberger condition $(S^4_5)$ \cite[p.113]{oko} for $s=5$.
This proves the claim. Finally, if $h^0(\e(-1))=0$ and
$h^0(\e_H(-1))=0$ for every hyperplane $H\subset\p^n$ then we get
$$h^0(\e)\leq h^0(\e_H)\leq\dots\leq h^0(\e_{\p^2})\leq
h^0(\e_{\p^1})=k+3.$$ Let us see that this is impossible. Consider
the exact sequence
$$0\to\mathcal K\to\oo_{\p^2}^{\oplus h^0(\e_{\p^2})}\to\e_{\p^2}\to 0$$ where $\mathcal
K$ is a vector bundle on $\p^2$ with $h^0(\mathcal K)=h^1(\mathcal
K)=0$, $c_1(\mathcal K)=-3$ and $c_2(\mathcal K)=c_2(\mathcal
K^*)=9-c_2\geq 5$. Then the Hirzebruch-Riemann-Roch theorem
$$\chi(\mathcal K)=\frac{1}{2}(c_1(\mathcal K)^2-2c_2(\mathcal
K)+3c_1(\mathcal K))+rk(\mathcal K)$$ for vector bundles on $\p^2$
yields $$0\leq h^2(\mathcal K)=-c_2(\mathcal K)+h^0(\e_{\p^2})-k\leq
-5+h^0(\e_{\p^2})-k$$ i.e. $h^0(\e_{\p^2})\geq k+5$, so we get a
contradiction.
\end{proof}

As a consequence, we obtain the classification of globally generated
vector bundles $\e$ on $\p^n$ with $c_1=3$ and no restriction on
$c_2$.

\begin{proof}[Proof of Corollary \ref{cor:main}]
The proof follows from Theorem \ref{thm:main} and Lemmas \ref{lem:1st reduction} and \ref{lem:2nd
reduction}.
\end{proof}

\begin{rem}\label{rem:hidden}
Some well-known globally generated vector bundles seem to be \emph{hidden} in Theorem \ref{thm:main}(viii) (e.g., $T_{\p^2}$) and Corollary \ref{cor:main}(viii) (e.g., the Tango bundle $\mathcal T$ given by the exact sequence $0\to T_{\p^4}(-2)\to\oo_{\p^4}^{\oplus 7}\to\mathcal T\to 0$; see, for instance, \cite[Ch.~I, $\S4$]{oko}). They can be easily detected in our classification by means of \cite[Ch.~I, Lemmas 4.3.1 and 4.3.2]{oko}. In this context, we point out that the only globally generated vector bundle of rank $k$ on $\p^n$ with $c_1=3$ and $k<n$ which does not split is the Tango bundle $\mathcal T$ of rank $3$ on $\p^4$, as one immediately deduces from Theorem \ref{thm:main} and Corollary \ref{cor:main} that $c_n(\e)=0$ if and only if $\e=\Omega_{\p^4}^2(2)^*\cong\wedge^2T_{\p^4}(-2)$, giving the following diagram:

\[
\xymatrix
{
        &                                 & 0                                 \ar[d] & 0 \ar[d]                            &    \\
        &                                 & \oo_{\p^4}^{\oplus 3}   \ar@{=}[r]\ar[d] & \oo_{\p^4}^{\oplus 3}        \ar[d] &    \\
0\ar[r] & T_{\p^4}(-2)   \ar[r]\ar@{=}[d] & \oo_{\p^4}^{\oplus 10}      \ar[r]\ar[d] & \wedge^2T_{\p^4}(-2)   \ar[r]\ar[d] & 0  \\
0\ar[r] & T_{\p^4}(-2)   \ar[r]           & \oo_{\p^4}^{\oplus 7}       \ar[r]\ar[d] & \mathcal T\ar[r]\ar[d]              & 0  \\
        &                                 & 0                                        & 0                                   &    \\
}
\]
\end{rem}

\begin{rem}
As in \cite{su2}, one can easily deduce the classification of triple Veronese embeddings of $\p^r$ in a Grassmannian of $(k-1)$-planes from Theorem \ref{thm:main} and Corollary \ref{cor:main}. The case $k=2$ has been studied in \cite{huh}. Globally generated vector bundles and embeddings in Grassmannians are closely related to matrices of constant rank on projective spaces (see \cite{m-m} and \cite{f-m}), but we will not consider this matter in this note.
\end{rem}

\begin{rem}
Following the research initiated in \cite{su2}, globally generated vector bundles and reflexive sheaves with low first Chern class on projective spaces and quadric hypersurfaces have recently been studied by several authors (see \cite{c-e}, \cite{ellia}, \cite{man}, \cite{a-m}, \cite{b-h-m1}, \cite{b-h-m2} and \cite{b-h-m3}).
\end{rem}


\begin{thebibliography}{1}

\bibitem{a-m}
C. Anghel, N. Manolache: Globally generated vector bundles on $\p^n$ with $c_1=3$. arXiv:1202.6261v1 [math.AG]

\bibitem{b-h-m1}
E. Ballico, S. Huh, F. Malaspina: Globally generated vector bundles of rank 2 on a smooth quadric threefold. arXiv:1211.1100 [math.AG]

\bibitem{b-h-m2}
E. Ballico, S. Huh, F. Malaspina: On higher rank globally generated vector bundles over a smooth quadric threefold. arXiv:1211.2593 [math.AG]

\bibitem{b-h-m3}
E. Ballico, S. Huh, F. Malaspina: Reflexive and spanned sheaves on $\mathbb{P}^3$. arXiv:1301.1822 [math.AG]

\bibitem{c-e}
L. Chiodera, Ph. Ellia: Rank two globally generated vector bundles with $c_1\leq 5$. Rend. Istit. Mat. Univ. Trieste 44 (2012) 413–-422.

\bibitem{ellia}
Ph. Ellia: Chern classes of rank two globally generated vector bundles on $\p^2$. arXiv:1111.5718 [math.AG]

\bibitem{f-m}
M.L. Fania, E. Mezzetti: Vector spaces of skew-symmetric matrices of constant rank. Linear Algebra Appl. 434 (2011) 2383--2403.

\bibitem{huh}
S. Huh: On triple Veronese embeddings of $\p^n$ in the Grassmannians. Math. Nachr. 284 (2011) 1453--1461.

\bibitem{m-m}
L. Manivel, E. Mezzetti: On linear spaces of skew-symmetric matrices of constant rank. Manuscripta Math. 117 (2005) 319--331.

\bibitem{man}
N. Manolache: Globally generated vector bundles on $\p^3$ with $c_1=3$. arXiv:1202.5988v1 [math.AG]

\bibitem{oko}
Ch. Okonek, M. Schneider, H. Spindler: Vector bundles on complex
projective spaces. Progress in Mathematics, 3. Birkh\"auser, Boston,
Mass., 1980.

\bibitem{su2}
J.C. Sierra, L. Ugaglia: On globally generated vector bundles on
projective spaces. J. Pure Appl. Algebra 213 (2009) 2141--2146.
\end{thebibliography}
\end{document}